\documentclass{amsart}

\usepackage{amsmath}
\usepackage{amssymb}
\usepackage{mathrsfs}
\usepackage{amsthm}
\usepackage{comment}

\title{An extension of $z$-ideals and $z^\circ$-ideals}
\author[A. R. Aliabad]{A. R. Aliabad}
\address{Department of Mathematics, Shahid Chamran University of Ahvaz, Ahvaz, Iran}
\email{aliabady\_r@scu.ac.ir}

\author[M. Badie]{M. Badie$^*$}
\address{Department of Basic Sciences, Jundi-Shapur University of Technology, Dezful, Iran}
\email{badie@jsu.ac.ir}

\author[S. Nazari]{S. Nazari}
\address{Department of Mathematics, Shahid Chamran University of Ahvaz, Ahvaz, Iran}
\email{sajadnazari1362@yahoo.com}



\theoremstyle{plain}
\newtheorem{Thm}{Theorem}[section]
\newtheorem{Lem}[Thm]{Lemma}
\newtheorem{Def}[Thm]{Definition}
\newtheorem{Pro}[Thm]{Proposition}
\newtheorem{Cor}[Thm]{Corollary}
\theoremstyle{definition}
\newtheorem{Exa}[Thm]{Example}
\newtheorem{Rem}[Thm]{Remark}

\newcommand{\cal}{\mathcal}

\newcommand{\N}{\mathbb{N}}
\newcommand{\R}{\mathbb{R}}
\newcommand{\Z}{\mathbb{Z}}

\newcommand{\Ge}[1]{\langle #1 \rangle}
\renewcommand{\H}{\mathcal{H}}
\newcommand{\J}{\text{\textbf{J}}}
\newcommand{\ff}{if and only if }
\newcommand{\RTA}{\Rightarrow}
\newcommand{\LTA}{\Leftarrow}
\newcommand{\LRTA}{\Leftrightarrow}
\newcommand{\TOO}{\rightarrow}
\newcommand{\Mi}{\mathrm{Min}}
\newcommand{\Ma}{\mathrm{Max}}
\newcommand{\An}{\mathrm{Ann}}
\newcommand{\Sp}{\mathrm{Spec}}
\newcommand{\Rad}{\mathrm{Rad}}
\newcommand{\Jac}{\mathrm{Jac}}
\newcommand{\sub}{\subseteq}

\newcommand{\set}{\setminus}
\newcommand{\ov}{\overline}

\newcommand{\la}{\lambda}
\newcommand{\La}{\Lambda}

\newcommand{\tohi}{\emptyset}

\newcommand{\ex}{\exists}

\begin{document}

\let\thefootnote\relax\footnote{$^*$ Corresponding author}

\begin{abstract}
Let $R$ be a commutative ring, $Y\sub \Sp(R)$ and $ h_Y(S)=\{P\in Y:S\sub P \}$, for every $S\sub R$. An ideal $I$ is said to be an $\H_Y$-ideal whenever it follows from $h_Y(a)\sub h_Y(b)$ and $a\in I$ that $b\in I$. A strong  $\H_Y$-ideal is defined in the same way by replacing an arbitrary finite set $F$ instead of the element $a$. In this paper these two classes of ideals (which are based on the spectrum of the ring $R$ and are a generalization of the well-known concepts semiprime ideal, z-ideal, $z^{\circ}$-ideal (d-ideal), sz-ideal and $sz^{\circ}$-ideal ($\xi$-ideal)) are studied. We show that the most important results about these concepts, ``Zariski topology", ``annihilator" and etc can be extended in such a way that the corresponding consequences seems to be trivial and useless. This comprehensive look helps to recognize the resemblances and differences of known concepts better.
\end{abstract}

\maketitle

\section{Introduction}

The concept of $z$-ideal, first studied in the rings of continuous functions as an ideal $I$ of $C(X)$ that $Z(f)\sub Z(g)$ and $f\in I$ implies that $g\in I$, see \cite{gillman1960rings}. Then this concept studied more generally for the commutative rings, in \cite{mason1973z}, as an ideal $I$ of $R$ that whenever two elements of $R$ are contained in the same family of maximal ideals and $I$ contains one of them, then it follows that $I$ contains the other one. If we  use $(Z(f))^{\circ}\sub (Z(g))^{\circ}$ instead of the above inclusion relation and the minimal prime ideals instead of the maximal ideals in the above definitions,  then we obtain the concept of $z^{\circ}$-ideal ($d$-ideal) in $C(X)$ and the commutative rings, which are introduced and carefully studied in 
\cite{azarpanah1999z,azarpanah2000ideals,huijsmans1980z}. The concepts of $z$-ideal and $z^{\circ}$-ideal can be generalized to the concepts of $sz$-ideal and $sz^{\circ}$-ideal ($\xi$-ideal), respectively, based on the finite subsets of the ideals instead of the single points in the ideal, and are studied in \cite{aliabad2011sz,Artico1981,mason1973z}.  

In this paper, we define and carefully study the $\H_Y$-ideals and the strong $\H_Y$-ideals which are a generalization of all of the above concepts. It is not difficult to see that a large amount of the results of the above mentioned papers and generally the papers in the literature about these topics, are special cases of the results of this paper. 

In the next section we recall some pertinent definitions. In Section 3, we define, characterize and give examples of $\H_Y$-ideals, strong $\H_Y$-ideals and $Y$-Hilbert ideals and study their relations. We give new characterizations of $z^{\circ}$-ideals and $sz^{\circ}$-ideals. It is shown that the minimal prime ideals over a (strong) $\H_Y$-ideal is again  (strong) $\H_Y$-ideals and so every (strong) $ \H_Y $-ideal is the intersection of minimal prime (strong) $\H_Y$-ideals containing it. In $C(X)$ the concepts of $\H_Y$-ideals and strong $\H_Y$-ideals coincide and the conditions under which these two classes of ideals coincide in an arbitrary ring is also considered in this section. The family of all $h_Y(F)$'s, where $F$ is an arbitrary finite subset of $R$, is closed under the finite intersection and union, hence it forms a distributive lattice. The study of (minimal prime, prime and maximal) filters of this distributive lattice and their correspondence with the (minimal prime, prime and maximal) strong $\H_Y$-ideals of $R$ is the subject of Section 4. Section 5 is devoted to propositions which generates a rich source of examples of $\H_Y$-ideals, strong $\H_Y$-ideals and $Y$-Hilbert ideals. For example if $I$ is a (strong) $\H_Y$-ideal, then $(J:I)$ and $I_A$ are (strong) $\H_Y$-ideals, where $A$ is a multiplicatively closed subset of $R$ disjoint from $I$. Moreover we give characterizations of Von Neumann regular rings, according to the (strong) $\H_Y$-ideals. In Section 6 we answer the natural questions that arise about the product, contraction, extension and quotients of (strong) $\H_Y$-ideals and $Y$-Hilbert ideals. In the last section we characterize certain (strong) $\H_Y$-ideals over or contained in an arbitrary ideal. For example for every ideal $I$, the smallest (strong) $\H_Y$-ideal containing $I$ exists and is shown by $I_{\H}$ ($I_{S\H}$). We give a precise characterization of these ideals and their properties.

\section{Preliminaries}

In this article, any ring $ R $ is commutative with unity.  A semiprime ideal is an ideal which is an intersection of prime ideals. The set of all ideals of $R$ is denoted by ${\cal{I}}(R)$. For each ideal $I\in{\cal{I}}(R)$ and each element $ a $ of $ R $, we denote the ideal $ \{ x \in R : ax \in I\}$ by $ ( I : a ) $. When $ I = \Ge{0}  $ we write $ \An(a) $ instead of $(\Ge{0}:a)$ and call this the annihilator of $a$. A prime ideal $ P $ containing an ideal $I$ is said to be a minimal prime  over $ I $, if there is no any prime ideal strictly contained in $ P $ that contains $ I $. $\Sp(R)$, $\Mi(R)$, $\Ma(R)$, $\Rad(R)$ and $\Jac(R)$ denote the set of all prime ideals, all minimal prime     ideals, all maximal ideals of $R$  and their intersections, respectively. By $\Mi(I)$ we mean the set of minimal prime ideals over $I$. In fact $\Mi(\Ge{0})=\Mi(R)$. A ring $R$ is said to be reduced if  $\Rad(R)=\Ge{0}$. If $\Jac(R)=\Ge{0}$, then we call $R$ semiprimitive. The socle of a ring $R$ is the sum of all minimal ideals of $R$.

A prime ideal $ P $ is called a Bourbaki associated prime divisor of an ideal $ I $ if $ ( I‌ : x ) = P $, for some $ x \in R $. We denote the set of all Bourbaki associated prime divisors of an ideal $ I $ by $ {\cal{B}}(I) $. We use $ {\cal{B}}(R) $ instead of $ {\cal{B}}(\Ge{0}) $. A representation $ I = \bigcap_{P‌ \in {\cal{P}}} P $ of $ I $ as an intersection of prime ideals is called irredundant if no $ P \in {\cal{P}} $ may be omitted. Let $ I $ be a semiprime ideal, $ P_\circ \in \Mi(I) $ is called irredundant with respect to $ I $,  if $ I \neq \bigcap_{ P_\circ \neq P \in \Mi(I) } P$. If $ I $ is equal to the intersection of all irredundant with respect to $ I $, then we call $ I $ a fixed-place ideal, exactly, by \cite[Theorem 2.1]{aliabad2013fixed}, we have $ I = \bigcap {\cal{B}}(I) $.

In this paper, all $Y\sub \Sp(R)$ is considered by Zariski topology; i.e.,  by assuming as a base for the closed sets of $Y$, the sets $h_Y(a)$ where $h_Y(a)=\{ P\in Y:~a\in P\}$. Hence, closed sets of $Y$ are of the form $h_Y(I)=\bigcap_{a\in I}h_Y(a)=\{P\in Y:~ I\sub P\}$, for some ideal $I$ in $R$. Also, we set $h_Y^c(I)=Y\backslash h_Y(I)$. For any subset $S$ of $Y$, we show the kernel of $S$ by $k(S)=\bigcap_{P\in S}P$ and  we have $\overline{S}=cl_Y S=h_Yk(S)$. When $Y=\Sp(R)$, we omit the index $Y$ and when $Y=\Ma(R)$ ($Y=\Mi(R)$) we write  $M$ (m) instead of $Y$ in the index. By these notations, for every $S\sub R$, we can use the notations $kh_m(S)$ and $kh_M(S)$ instead of $P_S$ and $M_S$ (which is usually used in the context of $C(X)$), respectively. We use the following well-known lemma frequently, one may see \cite[Lemma 4.1]{kist1963minimal} or \cite[Proposition 2.9]{aliabad2016alpha} for the proof.

\begin{Lem}
	Let $R$ be a ring, $Y\sub \Sp(R)$ and $k(Y)=I$. Then  $(I:S)=kh_Y^c(S)$, for every $S\sub R$. In particular, if $k(Y)=\langle 0\rangle $, then $Ann(S)=kh_Y^c(S)$.
	\label{introduction}
\end{Lem}

Throughout the paper $C(X)$ (resp., $C^*(X)$)  is the ring of all (resp., bounded) real valued continuous functions on a Tychonoff space $X$. Suppose that $f\in C(X)$, we denote $f^{-1}\{0\}$ by $Z(f)$ and $X\set Z(f)$ by $Coz(f)$. Every subset of $X$ of the form $Z(f)$ (resp., $Coz(f)$), for some $f \in C(X)$ is called zero set (resp., cozero set). A space $X$ is called pseudocompact, if $C(X)=C^*(X)$. $\ov{Coz(f)}$ is called the support of $f$. The family of all function in $C(X)$ with compact (resp., pseudocompact) support is denoted by $C_K(X)$ (resp., $C_\psi(X)$).

Recall that if $L$ is a lattice, then $\emptyset\neq F\sub L$ is a filter if $F$ is closed under the finite meet and whenever $a\in F$ and $b\geq a$, then it follows that $b\in F$. A filter $F$  is called prime if for every $a,b\in L$, $a\vee b\in F$ implies that $a\in F$ or $b\in F$.

The reader is referred to \cite{atiyah1969introduction,sharp2000steps,stephen1970general,gillman1960rings,blythlattices,gratzer98} for undefined terms and notations.

\section{$\H_Y $-ideals, $\H_Y $-filters, strong $\H_Y $-ideals, $Y$-Hilbert ideals  and their characterizations}

First, for a set $A$, we contract that  ${\bf{F}}(A) $ is the set of all finite subsets of $A$. Recall that a ring of sets is a collection of subsets of some set $A$ which is closed under the finite unions and  intersections. A ring of sets is obviously a distributive lattice. Now, for a ring $R$, we denote the collection  $\{h_Y (F): F \in {\bf{F}}(R)\}=\{ h_Y (I): I \text{ is a finitely generated ideal of } R \} $ by $\H_Y $.  Since for arbitrary ideals $I$ and $J$ of $R$, $h_Y (I) \cap h_Y (J) = h_Y (I+J) $, $h_Y (I) \cup h_Y (J) = h_Y (IJ)$, $ h_Y (\langle 0 \rangle)=Y $ and $h_Y (R)=h_Y (\langle 1\rangle)=\tohi$, also since the sum and the product of two finitely generated ideals are finitely generated, $ {\cal{H}}_Y  $ is a ring of sets and so it is a bounded distributive lattice. We call a filter of the distributive lattice $ {\cal{H}}_Y  $ an $\H_Y $-filter on $Y$. Note that all prime $ {\cal{H}}_Y  $-filters  and all $ {\cal{H}}_Y  $-ultrafilters  are assumed to be proper filters. Now suppose that $ \mathscr{F} $ is an $\H_Y $-filter on $Y  \sub \Sp(R)$ and $I$ is an ideal of $R$. We denote $ \{ h_Y (S) : S \in {\bf{F}}(I) \} $ and $ \{ a \in R : h_Y (a) \in \mathscr{F} \} $ by $\H_Y (I) $ and  $\H_Y^{-1}(\mathscr{F})$, respectively.

\begin{Lem}
    Let $I$ be an ideal of a ring $R$, $\mathscr{F} $ be an $\H_Y $-filter on $ Y  \sub \Sp(R) $ and $F$ be a finite subset of $R$. The following statements hold.
    \begin{itemize}
    	\item[(a)] $ h_Y (F) \in \mathscr{F} $ \ff $ F \sub \H_Y^{-1}(\mathscr{F})$.
        \item[(b)] $\H_Y^{-1}(\mathscr{F})$ is an ideal of $R$.
        \item[(c)] $\H_Y(I)$ is an $\H_Y $-filter on $Y$.
	\end{itemize}
\label{finite subset of Y -1}
\end{Lem}
\begin{proof}
    (a $\RTA$). For every $ s \in F $, $ h_Y (F) \sub h_Y (s) $, thus $ h_Y (s) \in \mathscr{F} $ and therefore $ s \in \H_Y^{-1}(\mathscr{F})$, for every $s\in F$. Hence, $ F \sub \H_Y^{-1}(\mathscr{F})$.

    (a $\LTA$). For every $ s \in S $, $ h_Y (s) \in \mathscr{F} $. Since $ F $ is finite, $ h_Y (F) = \bigcap_{s \in F} h_Y (s) \in \mathscr{F} $.

    (b). Let $a,b\in \H_Y^{-1}(\mathscr{F})$ and $r\in R$. Clearly since $h_Y(a)\cap h_Y(b)\sub h_Y(a+b)$ and $h_Y(a)\sub h_Y(ra)$, we have  $a+b, ra\in \H_Y^{-1}(\mathscr{F})$.

    (c). Suppose that $h_Y(F_1),h_Y(F_2)\in \H_Y(I)$, where $F_1$ and $F_2$ are finite subsets of $I$. Clearly $F_1\cup F_2$ is a finite subset of $I$ and consequently $h_Y(F_1)\cap h_Y(F_2)=h_Y(F_1\cup F_2)\in\H_Y(I)$.  Suppose now that $F_1$ is a finite subset of $I$, $h_Y(F_1)\sub h_Y(F_2)$ where $F_2$ is a finite subset of $R$. Clearly, $F_1F_2=\{s_1s_2:s_1\in F_1 \text{ and } s_2\in F_2\}$ is a finite subset of $I$ and $h_Y(F_2)=h_Y(F_1)\cup h_Y(F_2)=h_Y(F_1F_2)$. Consequently $h_Y(F_2)\in\H_Y(I)$.
\end{proof}

Note that for a proper ideal $I$, $\H_Y(I)$ is not necessarily a proper $\H_Y$-filter; for example if $Y=\Mi(R)$ and a proper ideal $I$ contains a non zero-divisor then $\H_Y(I)=\H_Y$. By the way, it is easy to see that if $\Ma(R) \sub Y$ then $\H_Y(I)$ is a proper $\H_Y$-filter, for every proper ideal $I$. Now by the two next propositions we define and characterize $\H_Y$-ideals and strong $\H_Y$-ideals.

\begin{Pro}
    Let $R$ be a ring, $Y \sub \Sp(R)$ and $I$ be an ideal of $R$. Then the following are equivalent:
    \begin{itemize}
        \item[(a)] For every $a\in I$ and  $S\sub R$, it follows from $h_Y(a)\sub h_Y(S)$ that $S\sub I$.
        \item[(b)] For every $a\in I$ and  $S\sub R$, it follows from $h_Y(a)=h_Y(S)$ that $S\sub I$.
        \item[(c)] For every $a\in I$ and  $b\in R$, it follows from $ h_Y(a)=h_Y(b)$ that $b\in I$.
        \item[(d)] For every $a\in I$ and  $b\in R$, it follows from $h_Y(a)\sub h_Y(b)$ that $b\in I$.
        \item[(e)]  If $a\in I$, then $kh_Y(a)\sub I$.
        \item[(f)] For every $a\in I$ and  $S\sub R$, it follows from $kh_Y(S)\sub kh_Y(a)$ that $S\sub I$.
        \item[(g)] For every $a\in I$ and  $S\sub R$, it follows from $kh_Y(S)=kh_Y(a)$ that $S\sub I$.
        \item[(h)] For every $a\in I$ and  $b\in R$, it follows from $ kh_Y(b)=kh_Y(a)$ that $b\in I$.
        \item[(k)] For every $a\in I$ and  $b\in R$, it follows from $kh_Y(b)\sub kh_Y(a)$ that $b\in I$.
    \end{itemize}
\label{hyideal}
\end{Pro}
\begin{proof}
    (a) $\RTA$ (b) $\RTA$ (c). They are trivial.

    (c) $\RTA$ (d). We know that $h_Y(a)\cup h_Y(b)=h_Y(ab)$, so if $h_Y(a)\sub h_Y(b)$, then $h_Y(ab)=h_Y(b)$ and $ab\in I$, so $b\in I$.

    (d) $\RTA$ (e). It is readily seen that  $b\in kh_Y(a)$ \ff $h_Y(a)\sub h_Y(b)$, thus if $a\in I$, then by the assumption, we have $kh_Y(a)\sub I$.

    (e) $\RTA$ (f) $\RTA$ (g) $\RTA$ (h). They are trivial.

    (h) $\RTA$ (k). Knowing this fact that $kh_Y(a)\cap kh_Y(b)=k(h_Y(a)\cup h_Y(b))=kh_Y(ab)$, it is follows, by using the same technique as (c $\RTA$ d).

    (k) $\RTA$ (a). If $h_Y(a)\sub h_Y(S)$, then $h_Y(a)\sub h_Y(s)$, for every $s\in S$. Whence $kh_Y(s)\subseteq kh_Y(a)$, for every $s\in S$ and consequently $S\sub I$.
\end{proof}

\begin{Def}
    Let $R$ be a ring and $Y \sub \Sp(R)$. An ideal $I$ of $R$ is said to be  an $\H_Y$-ideal if it satisfies in the equivalent conditions  of Proposition \ref{hyideal}.
\end{Def}

\begin{Pro}
    Let $R$ be a ring, $Y \sub \Sp(R)$ and $I$ be an ideal of $R$. Then the following are equivalent:
    \begin{itemize}
        \item[(a)] For every finite subset $F$‌ of $I$ and every $S\sub R$, it follows from  $h_Y(F)=h_Y(S)$ that $S\sub I$.
        \item[(b)] For every finite subset $F$‌ of $I$ and every finite subset $G$ of $R$, it follows from  $h_Y(F)=h_Y(G)$ that $G\sub I$.
        \item[(c)] For every finite subset $F$‌ of $I$ and every finite subset $G$  of $R$, it follows from  $h_Y(F)\sub h_Y(G)$ that $G\sub I$.
        \item[(d)] It follows from  $h_Y(a)\in \H_Y(I)$  that $a\in I$.
        \item[(e)] For every finite subset $F$ ‌of $R$, it follows from  $h_Y (F)\in \H_Y (I)$ that $F\sub I$.
        \item[(f)] For every finite subset $F$ of $I$ and $a\in R$‌, it follows from  $h_Y(F)=h_Y(a)$ that $a\in I$.
        \item[(g)] For every finite subset $F$ of $I$ and $a\in R$, it follows from  $h_Y(F)\sub h_Y(a)$ that $a\in I$.
        \item[(k)] For every finite subset $F\sub I$, we have $ kh_Y(F)\subseteq I$.
        \item[(l)] For every finite subset $F$ of $I$ and $a\in R$‌, it follows from  $kh_Y(a)=kh_Y(F)$ that $a\in I$.
        \item[(m)] For every finite subset $F$ of $I$ and $a\in R$, it follows from  $kh_Y(a)\sub kh_Y(F)$ that $a\in I$.
        \item[(n)] For every finite subset $F$ of $I$ and any $S\sub R$‌, it follows from  $kh_Y(S)=kh_Y(F)$ that $S\sub I$.
        \item[(o)] For every finite subset $F$ of $I$ and any $S\sub R$, it follows from  $kh_Y(S)\sub kh_Y(F)$ that $S\sub I$.
    \end{itemize}
\label{strong}
\end{Pro}
\begin{proof}
	By these facts that if $ A $  and $ B $ are arbitrary subsets of $ R $, then $ h_Y(AB) = h_Y(A) \cup h_Y(B) $ and $ B \subseteq kh_Y(A) $ \ff $ h_Y(B) \supseteq h_Y(A) $, it has a similar proof to the previous proposition.
\end{proof}

\begin{Def}
Let $R$ be a ring and $Y \sub \Sp(R)$. An ideal $I$ of $R$ is said to be a strong $\H_Y$-ideal if it satisfies in the equivalent conditions in Proposition \ref{strong}.
\end{Def}

\begin{Def}
    Suppose $Y  \sub \Sp(R)$. An ideal $I$ of $R$ is called a $Y$-Hilbert ideal, if $I$ is  an intersection of elements of some subfamily of $Y$; i.e., $I=kh_Y(I)$.
\end{Def}

Obviously, if $Y=\Ma(R)$, then the concepts of $\H_Y$-ideal, strong $\H_Y$-ideal and $Y$-Hilbert ideal coincide with the concepts of $z$-ideal, $sz$-ideal and Hilbert ideal in the literature, respectively, see \cite{aliabad2011sz} and \cite{mason1989prime}. Also, if $Y=\Mi(R)$, then the concepts of $\H_Y$-ideal and strong $\H_Y$-ideal coincide with the concepts of $z^\circ$-ideal (also known as d-ideal) and  $sz^\circ$-ideal (also known as $\xi$-ideal), respectively, see   \cite{aliabad2011sz}, \cite{azarpanah1999z}, \cite{azarpanah2000ideals}, \cite{Artico1981}, \cite{huijsmans1980z}, \cite{mason1973z}. Finally if $Y=\Sp(R)$, then the concepts of $\H_Y$-ideal, strong $\H_Y$-ideal, $Y$-Hilbert ideals and semiprime ideal coincide.  It is clear that every $Y$-Hilbert ideal is a strong $\H_Y $-ideal and every strong $\H_Y$-ideal is an $\H_Y$-ideal.  By the way, their converse does not hold generally even if $k(Y)=\langle 0\rangle$. If we set $Y=\Ma(C(X))$ then the ideal $O_0$ in $C(\R)$ is a strong $\H_Y$-ideal which is not intersection of maximal ideals. Moreover in \cite[Example 4.1]{aliabad2011sz} an example of a reduced ring is given which contains a $z^{\circ}$-ideal which is not a $sz^{\circ}$-ideal.

Clearly $kh_Y(F)$ is a strong $\H_Y$-ideal, for every finite set $F\sub R$, in fact, it is the smallest strong $\H_Y$-ideal containing $F$. In addition an ideal $I$ is a strong $\H_Y$-ideal \ff $I=\bigcup_{F\in {\bf{F}}(I)}kh_Y(F)=\sum_{F\in {\bf{F}}(I)}kh_Y(F)$. Also it is easy to see that if $X,Y\sub \Sp(R)$, then the family of strong $\H_X$-ideals and strong $\H_Y$-ideals coincide \ff $kh_X(F)=kh_Y(F)$, for every finite subset $F\sub R$. Note that in this case $k(X)=kh_X(0)=kh_Y(0)=k(Y)$, but the converse does not hold generally. For example $\Jac(C(X)) = \Ge{0} = \Rad(C(X))$ and the $z$-ideals and the $z^\circ$-ideals need not be coincide. Moreover, since $k(Y)=kh_Y(0)$, it follows that $k(Y)$ is the smallest strong $\H_Y$-ideal ($\H_Y$-ideal, Y-Hilbert ideal) in $R$.

Naturally, in this paper we were about to study other classes of ideals, close to $\H_Y$-ideals and strong $\H_Y$-ideals, using interior in the right-hand side of the inclusion in their definitions. For example, for $\H_Y$-ideal (strong $\H_Y$-ideal) case, it springs to mind to consider the ideals that it follows from $\left( h_Y(x)\right)^{\circ}\sub \left( h_Y(a)\right)^{\circ}$ ($\left( h_Y(F)\right)^{\circ}\sub \left( h_Y(a)\right)^{\circ}$) and $x\in I$ ($F\sub I$) that $a\in I$. But as one can observe below, we realized that if $Y\sub \Sp(R)$ and $k(Y)=\langle 0\rangle$, then these kind of ideals coincide with the $z^\circ$-ideals (resp., $sz^\circ$-ideals). The following lemma is an improvement of \cite[Proposition 1.1]{Artico1981}, without the redundant condition $\Mi(R)\sub Y$.

\begin{Lem}
\label{hYinterior}
	Let $Y\sub \Sp(R)$ and $k(Y)=\langle 0\rangle$. Then $\left( h_Y(S)\right)^{\circ}=h_Y^c\left( \An(S) \right)  $, for every $S\sub R$.
\end{Lem}
\begin{proof}
	By Lemma \ref{introduction}, we have 	
	\[h_Y(\An(S))=h_Y(k(h^c_Y(S)))=\overline{(h^c_Y(S))}=\left((h_Y(S))^\circ\right)^c.\]
	Consequently $(h_Y(S))^\circ=h^c_Y(\An(S))$.
\end{proof}

Suppose that $X,Y\sub \Sp(R)$. Clearly, $k(X)=k(Y)$ \ff $hk(X)=hk(Y)$; in the other words $\bigcap X=\bigcap Y$ \ff $\ov{X}=\ov{Y}$. Also, assume that $X$ is a topological space and dense in $T$. We know that if $W$ is an open subset of $T$, then $cl_T(W\cap X)=cl_TW$; equivalently, if $A$ is a closed subset of $T$, then $int_X(A\cap X)=(int_TA)\cap X$. By these facts we have the following lemma which is an improvement of \cite[Theorem 2.3]{Artico1981}.

\begin{Lem}
\label{kykx=0}
    Let $X,Y\sub \Sp(R)$ and $k(X)=\Rad(R)$. Then the following are equivalent:
    \begin{itemize}
    	\item[(a)] $k(Y)=\Rad(R)$.
    	\item[(b)] $\left( h_Y(S)\right)^{\circ}\sub h_Y(T)$ \ff $\left( h_X(S)\right)^{\circ}\sub h_X(T)$, for every  $T,S\sub R$.
    	\item[(c)] $\left( h_Y(S)\right)^{\circ}= \left( h_Y(T)\right)^{\circ}$  \ff $\left( h_X(S)\right)^{\circ}= \left( h_X(T)\right)^{\circ}$, for every  $T,S\sub R$.
    	\item[] If $k(Y)=\Ge{0}$, then the above statements are equivalent to the following statement.
    	\item[(d)] $kh_Y(S)\sub \An^2(S)$, for every $S\sub R$.
    \end{itemize}
\end{Lem}
\begin{proof} 
	Without loss of generality we can suppose that $X=\Sp(R)$.
	
    (a) $\RTA$ (b). Since $k(Y)=\Rad(R)$, it follows that $Y$ is dense in $X$  and so for every $S,T\sub Y$, we have
    \[ (h_X(S))^\circ\cap Y=(h_Y(S))^\circ\sub h_Y(T)\sub h_X(T) \]
    \[ \RTA~~(h_X(S))^\circ\sub\ov{(h_X(S))^\circ}\sub\ov{h_X(T)}=h_X(T).\]
    The converse is clear.

    (b) $\RTA$ (c). It is evident.

    (c) $\RTA$ (a).  Suppose  that $a\in k(Y)$. Since $\left( h_Y(a)\right)^{\circ}=Y=\left( h_Y(0)\right)^{\circ}$, it follows that $\left( h_X(a)\right)^{\circ}= \left(h_X(0)\right)^{\circ}=X$. Therefore, $h_X(a)=X$ and so $a\in  k(X)=\Rad(R)$.

    (a) $\LRTA$ (d). Since $k(Y)=\Ge{0}$, it is sufficient to show that (a) implies (d). For every $S\sub R$,
    \begin{align*}
    \An^2(S) & =kh_Y^ckh_Y^c(S)=k(Y\set h_Ykh_Y^c(S)) \\
    		 & =k(Y\set\ov{h_Y^c(S)})\supseteq k(Y\set h_Y^c(S))=kh_Y(S). \qedhere
    \end{align*}
\end{proof}

\begin{Lem}
	For every finite subset $F$ of $R$, we have $h_m(F)=\left(h_m(F)\right)^\circ$.
\label{hminterior}
\end{Lem}
\begin{proof}
	Suppose that $P\in\Mi(R)$, it easy to show that $F\sub P$ \ff $b\notin P$ exists such that $bF\sub\Rad(R)$. Then 
	\begin{align*}
	P\in h_m(F) & \LRTA\quad F\sub P\quad\LRTA\quad\exists b\notin P\quad bF\sub\Rad(R)\\
				& \LRTA\quad\exists b\in\left(\Rad(R):F\right)\setminus P\quad\LRTA\quad\left(\Rad(R):F\right)\not\sub P\\
				& \LRTA\quad P\notin h_m\left(\Rad(R):F\right) 
	\end{align*}
	Hence $h_m(F)=h^c_m\left(\Rad(R):F\right)$. Now with a method similar to Lemma \ref{hYinterior}, $\left(h_m(F)\right)^\circ=h^c_m\left(\Rad(R):F\right)$, hence $\left(h_m(F)\right)^\circ=h_m(F)$.
\end{proof}

By the above lemmas we give new characterizations of $z^{\circ}$-ideals and $sz^\circ$-ideals in the following proposition.

\begin{Pro}
	Let $Y\sub \Sp(R)$ and $k(Y)=\Rad(R)$. Then the following statements hold:
	\begin{itemize}
		\item[(a)] $I$ is a $z^{\circ}$-ideal \ff it follows from $\left(h_Y(b)\right)^{\circ}\sub h_Y(a)$ and $b\in I$ that $a\in I$; \ff it follows from $\left( h_Y(b)\right)^{\circ}\sub h_Y(S)$ and $b\in I$ that $S\sub I$.
		\item[(b)] $I$ is a $sz^\circ$-ideal \ff for every finite subset $F$ of $I$, it follows from $\left(h_Y(F)\right)^{\circ}\sub h_Y(a)$ that $a\in I$; \ff for every finite subset $F$ of $I$, it follows from $\left(h_Y(F)\right)^{\circ}\sub h_Y(S)$ that $S\sub I$.
\end{itemize}
\end{Pro}
\begin{proof}
    We prove one part and the other parts have similar proofs. By Proposition \ref{strong}, $I$ is a $sz^\circ$-ideal \ff for every finite subset $F$ of $I$, $h_m(F)\sub h_m(a)$ implies that $a\in I$; \ff for every finite subset $F$ of $I$, $\left(h_m(F)\right)^\circ\sub h_m(a)$ implies that $a\in I$, by Lemma \ref{hminterior}. Now Lemma \ref{kykx=0} concludes that this is equivalent to say, for every finite subset $F$ of $I$, it follows from $\left(h_Y(F)\right)^{\circ}\sub h_Y(a)$ that $a\in I$.
\end{proof}

Finally in the following improvement of \cite[Proposition 2.9]{aliabad2011sz}, \cite[Theorem 2.3]{Artico1981} and \cite[Proposition 2.12]{mason1973z}, we see the conditions under which every $z^{\circ}$-ideal ($sz^\circ$-ideal) is an $\H_Y$-ideal (a strong $\H_Y$-ideal).

\begin{Pro}
    If $Y\sub \Sp(R)$, then the following statements are equivalent:
    \begin{itemize}
    	\item[(a)] $k(Y)=\Rad(R)$.
    	\item[(b)] Every $z^{\circ}$-ideal is an $\H_Y$-ideal.
    	\item[(c)] Every $sz^{\circ}$-ideal is a strong $\H_Y$-ideal.
    	\item[(d)] $kh_Y(F)\sub kh_m(F)$, for every finite set $F\sub R$.
    	\item[(e)] $kh_Y(a)\sub kh_m(a)$, for every  $a\in R$.
    \end{itemize}
\end{Pro}
\begin{proof}
It has a same proof as \cite[Proposition 2.9]{aliabad2011sz}.
\end{proof}

We use the following lemma frequently.

\begin{Lem}
    Let $ Y  \sub \Sp(R) $. Every $\H_Y $-ideal is a semiprime ideal.
\label{semiprime}
\end{Lem}
\begin{proof}
    Suppose $x^n\in I$, so $h_Y(x)=h_Y(x^n)\in \H_Y(I)$. Thus $x\in I$.
\end{proof}

The following theorem and corollary show that the prime (strong) $\H_Y$-ideals play a vital role in the study of the (strong) $\H_Y$-ideals.

\begin{Thm}
    Let $Y \sub \Sp(R)$ and $I$ be a (strong) $\H_Y $-ideal. If $P \in \Mi(I)$, then $P$ is a (strong) $\H_Y $-ideal, too.
\label{minimal of I}
\end{Thm}
\begin{proof}
	From  Lemma \ref{semiprime}, it follows that $I$ is a semiprime ideal. Now suppose that $F$ is a finite subset of $P$, so there is some $b \notin P$ such that $bF \sub I$, thus $kh_Y (S) \cap kh_Y (b) = k\big( h_Y (F) \cup h_Y (b) \big) = kh_Y (bF) \sub I \sub P$. Since $kh_Y (b) \not\sub P$, it follows that $kh_Y (F) \sub P$. Consequently, by Proposition \ref{strong},  $P$ is a strong $\H_Y $-ideal. The other part has a similar proof.
\end{proof}

The above theorem concludes the following corollary, immediately.

\begin{Cor}\label{corollary}
    If $Y\sub \Sp(R)$, then the following statements hold:
    \begin{itemize}
        \item[(a)] An ideal $I$ is a (strong) $\H_Y$-ideals \ff it is an intersection of minimal prime (strong) $\H_Y$-ideals over $I$.
        \item[(b)] Every proper maximal (strong) $\H_Y$-ideal is a prime (strong) $\H_Y$-ideal.
    \end{itemize}
\end{Cor}

We turn  our attention now to considering the situations under which strong $\H_Y$-ideals and  $\H_Y$-ideals coincide. A ring $R$ is said to have the $h_Y$-property if for every $a,b \in R$, there is some $c \in R$ such that $h_Y(a) \cap h_Y(b) = h_Y(c)$. Clearly, this is equivalent to say that for any finite subset $F$ of $R$, there is some $c \in R$ such that $h_Y(F) = h_Y(c)$. Clearly, if  $Y\sub \Sp(R)$ and $R$ satisfies $h_Y$-property (for example if $R$ is B\'{e}zout domain), then the family of all $\H_Y$-ideals and the family of all strong $\H_Y$-ideals coincide. Also, the same fact is true in $C(X)$, since for every prime ideal $P$ of $C(X)$, ‌we have$f^2+g^2\in P$ \ff $f,g\in P$  and consequently $h_Y(f)\cap h_Y(g)=h_Y(f^2+g^2)$, for every $Y\sub \Sp(C(X))$ and every $f,g\in C(X)$. However, in Example \ref{converse is not true}, we show that the converse of this fact is not true.

Recall that a ring $R$ is said to satisfy annihilator condition (is called an a.c. ring), if for each finite set $F\sub R$ there is some $c\in R$ such that $\An(F)=\An(c)$. If $k(Y)=\Ge{0} $ and $R$ has $h_Y$-property, then $R$ is an a.c. ring. To see this, suppose $a,b\in R$ are given, then there exists some $c\in R$ such that $h_Y(a)\cap h_Y(b)=h_Y(c)$. Therefore using Lemma \ref{introduction} we have,
\[h_Y^c(a)\cup h_Y^c(b)=(h_Y(a)\cap h_Y(b))^c=h_Y^c(c)\RTA\]
\[kh_Y^c(a)\cap kh_Y^c(b)=k(h_Y^c(a)\cup h_Y^c(b))=kh_Y^c(c) \RTA Ann(a)\cap Ann(b)=Ann(c).\]
One can easily see that if $h_Y^c(a)$ is a closed set, for every $a\in R$, then the converse is also true, for example $\Mi(R)$ has this property, see \cite[Theorem 2.3]{henriksen1965space}.

Suppose that $Y\sub\Mi(R)$, then \cite[Theroems 2.2 and 2.3]{henriksen1965space} imply that $h^c_Y(F)$ is closed in $Y$, for every finite subset $F$ of $R$. Now, clearly, if $I$ is an arbitrary ideal of $R$, then the mapping $a\TOO a+I$ induces a homeomorphism from $\Mi(R/I)$ to $\Mi(I)$. Consequently, if $Y\sub\Mi(I)$, then $h^c_Y(F)$ is closed in $Y$, for every finite subset $F$ of $R$. Using this fact, we have the following proposition, which characterizes the $h_Y$-property  when $I=k(Y)$ and $Y\sub\Mi(I)$.

\begin{Pro}
    Let $ Y \sub \Sp(R) $ and $ I = k(Y) $. If $ Y \sub \Mi(I) $, then the following statements are equivalent.
    \begin{itemize}
		\item[(a)] $R$ has $h_Y$-property.
        \item[(b)] For all finite sets $F\sub R$, there is some $c \in R$ such that $(I:F) = (I:c)$.
        \item[(c)] For every  $a,b\in R$, there is some $c \in R$ such that $(I:a) \cap (I:b) = (I:c)$.
        \item[(d)] $ R/I $ is an a.c. ring.
\end{itemize}
\label{H_Y and strong H_Y coincide in Minp(R)}
\end{Pro}
\begin{proof}
	(a) $\RTA$ (b). Let $F$ be a finite subset of $R$ and $h_Y(F)=h_Y(c)$ for some $c\in R$, then by Lemma \ref{introduction}, $(I:F)=kh_Y^c(F)=kh_Y^c(c)=(I:c)$.

    (b) $\RTA$ (a). Let $F$ be a finite subset of $R$. Since $Y\sub\Mi(I)$, $h_Y^c(F)$ and $h_Y^c(c)$ are closed sets, using Lemma \ref{introduction} and the assumption we have
    \[ h_Y^c(F)=h_Ykh_Y^c(F)=h_Y((I:F))=h_Y((I:c))=h_Ykh_Y^c(c)=h_Y^c(c).\]
    Consequently $h_Y(F)=h_Y(c)$.

	(b) $\LRTA$ (c). Since $(I:A)\cap (I:B)=(I:A\cup B)$, for every $A,B\sub R$,  it is evident.
	
    (c) $\LRTA$ (d). Clearly, for every $x\in R$, we have $\An(x+I)=\frac{(I:x)}{I}$. Therefore, we can write 
    \[\An(a+I)\cap\An(b+I)=\An(c+I)~~\LRTA~~\frac{(I:a)}{I} \cap \frac{(I:b)}{I}=\frac{(I:c)}{I}\]
    \[\LRTA~~\frac{(I:a) \cap (I:b) }{I}= \frac{(I:c)}{I}~~\LRTA~~(I:a)\cap(I:b)=(I:c).\]
    \end{proof}

\begin{Cor}
    Let $ Y \sub \Sp(R) $ and $ I = k(Y) $. If one of the following conditions holds, then the family of all $\H_Y$-ideals and the family of all strong $\H_Y$-ideals coincide.
    \begin{itemize}
    	\item[(a)] $Y$ is a fixed-place family and $R/I$ is an a.c. ring.
    	\item[(b)] $Y$ is a strong fixed-place family.
    \end{itemize}
\label{H_Y and strong H_Y coincide in fixed-place}
\end{Cor}
\begin{proof}
	(a). Since $Y$ is a fixed-place family, $ Y = {\cal{B}}(I) \sub \Mi(R) $, by \cite[Theorem 2.1]{aliabad2013fixed}. Now Proposition \ref{H_Y and strong H_Y coincide in Minp(R)}, completes the proof.
	
	(b). It follows immediately form \cite[Theorem 2.10, Corollary 2.11 and Theorem 2.10]{aliabad2017Bourbarki} and part (a).
\end{proof}

\begin{Exa}
	Suppose that $ X = \Sp(R) $ and $ Y = \Mi(R) $. In \cite[Example 3.3]{henriksen1965space}, a ring $ R $  is given  which does not satisfy annihilator condition, so $ R $ does not satisfy $ h_Y $-property, by Proposition \ref{H_Y and strong H_Y coincide in Minp(R)}. Thus $ R $  does not satisfy  $ h_X $-property, whereas the family of all $ \H_X $-ideals coincides with the family all strong $ \H_X $-ideals.
\label{converse is not true}
\end{Exa}

\section{Correspondence between $\H_Y $-filters and strong $\H_Y $-ideals}

In this section  we study the relation  and the correspondence between the strong $\H_Y$-ideals, the $\H_Y$-ideals and the $H_Y$-filters. First recall that if $E$ and $F$ are two partial ordered sets, then an order preserving mapping $f:E\TOO F$ is said to be residuated whenever there exists an order preserving mapping $g:F\TOO E$ such that $id_E\leq gf$ and $id_F\geq fg$; moreover, $g$ is unique and it is called the residual of $f$. The set of all $\H_Y$-filters on $Y\sub \Sp(R)$ is denoted by ${\cal{F}}_Y$.

In the following proposition we state the properties of the mappings $\H_Y$ and $\H^{-1}_Y$ and the image and preimage of ideals and filters under them, respectively.

\begin{Pro}
    Let $Y  \sub \Sp(R)$, $I\in {\cal{I}}(R)$ and  $\mathscr{F}\in {\cal{F}}_Y$. The following statements hold.
    \begin{itemize}
        \item[(a)] $\H_Y^{-1}(\mathscr{F})=R$ \ff $\mathscr{F}=\H_Y$.
        \item[(b)]  $I \sub \H_Y^{-1}\H_Y(I)$ and $\H_Y\H_Y^{-1}(\mathscr{F})=\mathscr{F}$.
        \item[(c)] $\H_Y$ is a residuated mapping from ${\cal{I}}(R)$  to ${\cal{F}}_Y$, and  $\H^{-1}_Y$ is the residual of $\H_Y$. Consequently, $\H_Y\H_Y^{-1}\H_Y=\H_Y$ and $\H_Y^{-1}\H_Y\H_Y^{-1}=\H_Y^{-1}$.
        \item[(d)] $I$ is a strong $\H_Y$-ideal \ff $I={\cal{H}}_Y^{-1}\H_Y (I)$.
        \item[(e)]  $\H_Y^{-1}(\mathscr{F})$ is a strong $\H_Y $-ideal of $R$.
        \item[(f)] If $\Ma(R)\sub Y  $, then $\H_Y (I)$ is a proper $\H_Y $-filter on $Y $, for every proper ideal $I$ of $R$.
        \item[(g)] If $I$ is a proper strong $\H_Y $-ideal of $R$, then $\H_Y (I)$ is a proper $\H_Y $-filter on $Y $.
    \end{itemize}
\label{Primary property}
\end{Pro}
\begin{proof}
    (a). $\H_Y^{-1}(\mathscr{F})=R~~\LRTA~~1\in\H_Y^{-1}(\mathscr{F})~~\LRTA~~h_Y(1)\in\mathscr{F}~~\LRTA~~\emptyset\in\mathscr{F}~~\LRTA~~\mathscr{F}=\H_Y.$

	(b). The first part is readily verified. Using Lemma \ref{finite subset of Y -1}(a), for every finite subset $F$ of $R$ we have
	\[  h_Y (F) \in \mathscr{F} \LRTA F \sub \H_Y^{-1}(\mathscr{F}) \LRTA h_Y (F)\in \H_Y\H_Y^{-1}(\mathscr{F}). \]

	(c). By part (b) and Lemma \ref{finite subset of Y -1}, the first part is trivial. For the second part see  \cite[Theorem1.5]{blythlattices}.

	(d). Let $I$ be a strong $\H_Y$-ideal. If $a\in \H_Y^{-1}\H_Y(I)$, then $h_Y(a)\in \H_Y(I)$ and so by Proposition \ref{strong}, $a\in I$. Now by  part (b), $I={\cal{H}}_Y^{-1}\H_Y (I)$. Conversely, suppose that $h_Y(a)\in \H_Y(I)$, whence $a\in \H_Y^{-1}\H_Y(I)=I$, therefore, by Proposition \ref{strong}, $I$ is a strong $\H_Y$-ideal.

	(e). Clearly, by part (c) we have $\H_Y^{-1}\H_Y\H_Y^{-1}(\mathscr{F})=\H_Y^{-1}(\mathscr{F})$, for every $\H_Y$-filter $\mathscr{F}$ on $Y$ and thus by part (d), $\H_Y^{-1}(\mathscr{F})$ is a strong $\H_Y$-ideal of $R$.

	(f). On the contrary, let $\emptyset \in \H_Y(I)$, then $ \emptyset = h_Y(F)$, for some finite set $F\sub I$, now by hypothesis $\Ge{F}=R$, which is a contradiction.

	(g). Since $R\ne I=\H_Y^{-1}\H_Y(I)$,  by part (a), it follows that $\H_Y(I)$ is a proper $\H_Y$-filter.
\end{proof}

The following corollary is an immediate consequence of the above proposition which gives  a correspondence between strong $\H_Y$-ideals and $H_Y$-filters.

\begin{Cor}
\label{coro 2.16}
    The following facts hold.
    \begin{itemize}
        \item[(a)] Suppose $ \mathscr{F} $ and $ \mathscr{G} $ are two $ \H_Y  $-filters on $Y  \sub \Sp(R)$. Then $ \mathscr{F} = \mathscr{G} $ \ff $\H_Y^{-1} (\mathscr{F}) =\H_Y^{-1} (\mathscr{G} ) $.
        \item[(b)] If $I$ and $J$ are two strong $\H_Y$-ideals then $\H_Y(I)=H_Y(J)$ \ff $I=J$.
        \item[(c)] The mapping $\H_Y$ is an order isomorphism from the set of all strong $\H_Y$-ideals onto the set of all $\H_Y$-filters on $Y$.
    \end{itemize}
\end{Cor}

In the following theorem we try to present  a correspondence between prime (maximal) strong $\H_Y$-ideals and prime (maximal) $H_Y$-filters.

\begin{Thm}
    Let $Y  \sub \Sp(R)$, $I\in {\cal{I}}(R)$ and  $\mathscr{F}\in {\cal{F}}_Y$. The following statements hold.
    \begin{itemize}
        \item[(a)] $\H_Y^{-1}(\mathscr{F})$ is a prime strong  $\H_Y $-ideal \ff $ \mathscr{F} $ is a prime $\H_Y $-filter.
		\item[(b)] If $I$ is a strong  $\H_Y $-ideal, then $I$ is a prime ideal of $R$ \ff $\H_Y (I)$ is a prime $\H_Y $-filter.
		\item[(c)] The mapping $\H_Y$ is one-to-one from the set of all prime strong $\H_Y$-ideals onto the set of all prime $\H_Y$-filters.
		\item[(d)] An ideal $I$ of $R$ is a maximal proper strong $\H_Y$-ideal \ff there exists an $\H_Y$-ultrafilter $\mathscr{F}$ such that $I=\H_Y^{-1}(\mathscr{F})$. In addition the mapping $\H_Y$ is one-to-one from the set of all maximal proper strong $\H_Y$-ideals onto the set of all  $\H_Y$-ultrafilters.
		\item[(e)] Assume that $\Ma(R)\sub Y$. If $I$ is a maximal ideal, then $\H_Y(I)$  is an $\H_Y$-ultrafilter. Supposing $I$ is a strong $\H_Y$-ideal, the converse is also true.
        \item[(f)] If $ \Ma(R) \sub Y  $, then $ \mathscr{F} $ is an $\H_Y $-ultrafilter \ff $\H_Y^{-1}(\mathscr{F})$ is a maximal ideal.
    \end{itemize}
\label{maxiaml storng and ultra}
\end{Thm}
\begin{proof}
    (a $\RTA$). Clearly, by Proposition \ref{Primary property}, $\H_Y^{-1}(\mathscr{F})$ is a proper ideal \ff $ \mathscr{F} $ is a proper $\H_Y$-filter. Now, suppose that $F_1$ and $F_2$ are two finite subsets of $R$ and  $ h_Y (F_1) \cup h_Y (F_2) \in \mathscr{F} $, then $ h_Y (F_1 F_2) \in \mathscr{F} $, so $ F_1 F_2 \sub \H_Y^{-1}(\mathscr{F}) $, by Lemma \ref{finite subset of Y -1}. Thus, either $ F_1 \sub \H_Y^{-1}(\mathscr{F})$ or $ F_2 \sub \H_Y^{-1}(\mathscr{F})$ and therefore either $ h_Y (F_1) \in \mathscr{F} $ or $ h_Y (F_2) \in \mathscr{F} $. Hence $ \mathscr{F} $ is a prime $\H_Y $-filter.

    (a $\LTA$). Suppose $ a b \in \H_Y^{-1}(\mathscr{F})$, then $ h_Y (ab) \in \mathscr{F} $, so $ h_Y (a) \cup h_Y (b) \in \mathscr{F} $, thus either $ h_Y (a) \in \mathscr{F} $ or $ h_Y (b) \in \mathscr{F} $, and therefore either $ a \in \H_Y^{-1}(\mathscr{F})$ or $ b \in \H_Y^{-1}(\mathscr{F})$. Hence, $\H_Y^{-1}(\mathscr{F})$ is a prime ideal.

    (b). It can be obtained easily by the previous part and Proposition \ref{Primary property}.

    (c). It is straightforward by using parts (a) and (b) as well as Corollary \ref{coro 2.16}.

    (d $\RTA$). Assume that $I$ is a maximal proper strong $\H_Y$-ideal. Clearly by Proposition \ref{Primary property}, $\H_Y(I)$ is a proper $\H_Y$-filter and so there exists an $\H_Y$-ultrafilter $\mathscr{F}$ such that $\H_Y(I)\sub\mathscr{F}$. Thus, $I=\H_Y^{-1}\H_Y(I)\sub\H_Y^{-1}(\mathscr{F})$ and since $\H_Y^{-1}(\mathscr{F})$ is a proper strong $\H_Y$-ideal, it follows that $I=\H_Y^{-1}(\mathscr{F})$.

    (d $\LTA$). Assume that $I=\H_Y^{-1}(\mathscr{F})$ where $\mathscr{F}$ is an $\H_Y$-ultrafilter. Clearly by Proposition \ref{Primary property}, $I$ is a proper strong $\H_Y$-ideal. Now suppose that $J$ is a proper strong $\H_Y$-ideal containing $I$. Thus by Proposition \ref{Primary property}, $\H_Y(J)$ is a proper $\H_Y$-filter containing $\H_Y(I)=\mathscr{F}$, so $\H_Y(I)=\H_Y(J)$, whence by Corollary \ref{coro 2.16}, we have $I=J$. The second part of (d) is straightforward.
    
    (e). Knowing this fact that if $\Ma(R)\sub Y$, then the maximal proper strong $\H_Y$-ideals are exactly the elements of $\Ma(R)$, this part follows easily from the previous part.

    (f). It is clear from the previous part.
\end{proof}

Since $\H_Y $ is a distributive lattice and a filter is a dual of an ideal, clearly, we have the following facts.

\begin{Pro}
    Suppose $ Y  \subseteq \Sp(R)$, $ \mathscr{F} $ is an $\H_Y $-filter on $Y $ and $ {\cal{S}} $ is a $\cup$-closed subset of $\H_Y $. If $\mathscr{F} \cap {\cal{S}} = \emptyset $, then there is a prime $\H_Y $-filter $\mathscr{P}$ containing $ \mathscr{F} $ such that $ \mathscr{P} \cap {\cal{S}} = \emptyset $.
\end{Pro}

\begin{Def}
    An $\H_Y $-filter $\mathscr{P}$ is called a minimal prime $\H_Y $-filter over a $\H_Y $-filter $\mathscr{F}$,  if there are no prime $\H_Y $-filter strictly contained in $\mathscr{P}$ that contains $\mathscr{F}$. By  $\Mi(\mathscr{F})$ we mean the set of all minimal prime $\H_Y $-filters over $\mathscr{F}$.
\end{Def}

The following corollary is an immediate consequence of Lemma \ref{finite subset of Y -1} and Theorem \ref{maxiaml storng and ultra}.

\begin{Cor}
    Let $ Y  \sub \Sp(R) $. Every $\H_Y $-filter $\mathscr{F}$ is the intersection of all minimal prime $\H_Y $-filters over $\mathscr{F}$
    \label{semiprime filter}
\end{Cor}

By this fact that for each semiprime ideal $ I $, $ P \in \Mi(I) $ \ff for each $ a \in P $, there is some $ b \notin P $ such that $ ab \in I $, the following proposition and corollary conclude from Theorem \ref{finite subset of Y -1} and the previous corollary.

\begin{Pro}
    Let $\mathscr{F}$ be an $\H_Y $-filter. $\mathscr{P} \in \Mi(\mathscr{F})$ \ff for every $ A \in \mathscr{P}$ there is some $ B \in \H_Y   \setminus \mathscr{P}$ such that $ A \cup B \in \mathscr{F} $.
\end{Pro}

\begin{Cor}
    $ \mathscr{P} \in \Mi(\{Y\}) $ \ff for every $ A \in \mathscr{P} $ there is a $ B \notin \mathscr{P} $ such that $ A \cup B = Y  $.
    \label{Minimal prime filter}
\end{Cor}

\begin{Pro}
    Let $Y  \sub \Sp(R)$ and $\mathscr{F}$ and $\mathscr{P}$ are two $\H_Y$-filters. Then the following statements hold
    \begin{itemize}
    	\item[(a)] $\mathscr{P} \in \Mi(\mathscr{F})$ \ff  $\H_Y^{-1}(\mathscr{P}) \in \Mi(\H_Y^{-1}(\mathscr{F}))$.
		\item[(b)] If $I$ is a strong $\H_Y$-ideal, then $P\in\Mi(I)$ \ff $\H_Y(P)\in\Mi(\H_Y(I))$.
	\end{itemize}
\label{Minimal filter and ideal}
\end{Pro}
\begin{proof}
	(a $\RTA$). Let $P_\circ$ be a minimal prime ideal over the strong $\H_Y$-ideal $\H_Y^{-1}(\mathscr{F})$ such that $\H_Y^{-1}(\mathscr{F})\sub P_\circ \sub\H_Y^{-1}(\mathscr{P})$. By Theorem \ref{minimal of I}, $P_\circ$ is a strong $\H_Y$-ideal and hence by Theorem \ref{maxiaml storng and ultra}, $\H_Y(P_\circ)$ is a prime $\H_Y$-filter such that $\mathscr{F}\sub\H_Y(P_\circ)\sub\mathscr{P}$. Therefore, $\H_Y(P_\circ)=\mathscr{P}$ and so $\H_Y^{-1}(\mathscr{P})=P_\circ$, by Corollary \ref{coro 2.16}.
	
	(a $\LTA$). Assume that $\mathscr{P}_\circ$ is a prime $\H_Y$-filter such that $\mathscr{F}\sub\mathscr{P}_\circ \sub\mathscr{P}$. By Theorem \ref{maxiaml storng and ultra}, $\H_Y^{-1}(\mathscr{P}_\circ)$ is a prime strong $\H_Y$-ideal and $\H_Y^{-1}(\mathscr{F})\sub\H_Y^{-1}(\mathscr{P}_\circ)\sub\H_Y^{-1}(\mathscr{P})$. Therefore, $\H_Y^{-1}(\mathscr{P}_\circ)=\H_Y^{-1}(\mathscr{P})$ and so $\mathscr{P}_\circ=\mathscr{P}$, by Corollary \ref{coro 2.16}.
	
	(b $\RTA$). Assume that $I$ is a strong $\H_Y$-ideal and $P\in \Mi(I)$, then by Theorem \ref{minimal of I}, $P$ is  a strong $\H_Y$-ideal, so  $ \H_Y^{-1}\H_Y(P)\in \Mi(\H_Y^{-1}\H_Y(I)) $. Hence by part (a) and Theorem \ref{maxiaml storng and ultra}, $ \H_Y(P)\in \Mi(\H_Y(I)) $.
	
	(b $\LTA$). Let $Q\in\Mi(I)$ and $I\sub Q\sub P$. Hence, $ \H_Y(Q) $ is a prime $ \H_Y $-filter, by Theorem \ref{maxiaml storng and ultra}, and $ \H_Y(I) \subseteq \H_Y(Q) \subseteq \H_Y(P) $, so $ \H_Y(Q) = \H_Y(P) $. By Theorem \ref{minimal of I}, $ Q $ is a strong $ \H_Y $ -ideal, thus $P\sub\H_Y^{-1}\H_Y(P)=\H_Y^{-1}\H_Y(Q)=Q$ and so $Q=P$.
\end{proof}

\begin{Pro}
    Let $ Y  \sub \Sp(R) $, $ k(Y)  = I$ and $ \mathscr{F} $ be an $\H_Y $-filter on $Y $. Set $ T= \{ P / I : P \in Y  \} $ and for every $ A \in \mathscr{F} $ define $ A' = \{ P / I : P \in A \} $ and $ \mathscr{F}' = \{ A' : A \in \mathscr{F} \} $. The following statements hold.
    \begin{itemize}
        \item[(a)] $ \mathscr{F}' $ is an $\H_T$-filter on $T$.
        \item[(b)] $ \mathscr{F} $ is a prime $\H_Y $-filter on $Y $ \ff $ \mathscr{F}' $ is a prime $\H_T$-filter on $T$.
        \item[(c)] $ \mathscr{F} $ is an $\H_Y $-ultrafilter on $Y $ \ff $ \mathscr{F}' $ is an $\H_T$-ultrafilter on $T$.
        \item[(d)] $ \frac{\H_Y^{-1}(\mathscr{F})}{I} = \H_T^{-1}(\mathscr{F}')$.
    \end{itemize}
\end{Pro}
\begin{proof}
    It is straightforward.
\end{proof}

\begin{Pro}
    Suppose $R'$ is a subring of a ring $R$ and $Y \sub \Sp(R)$, then $Y' = \{ P \cap R' : P \in Y \} \sub \Sp(R') $. Set \[ \mathscr{F}' = \{ h_{Y'}(F) : h_Y(F) \in \mathscr{F} \text{ and } F \text{ is a finite subset of } R' \} \]
    for every $\H_Y$-filter $\mathscr{F}$. The following statements hold.
    \begin{itemize}
        \item[(a)] $h_{Y'}(S) = \{ P \cap R' : P \in h_Y(S) \} $, for every $S \sub R'$.
        \item[(b)] If $ \mathscr{F} $ is an $\H_Y$-filter, then $\mathscr{F}'$ is an $\H_{Y'}$-filter .
        \item[(c)] For every $\H_{Y'}$-filter $\mathscr{G}$, there is some $\H_Y$-filter $\mathscr{F}$ such that $\mathscr{G} = \mathscr{F}'$.
        \item[(d)] $\H^{-1}_Y(\mathscr{F}) \cap R' = \H^{-1}_{Y'}(\mathscr{F}')$, for every $\H_Y$-filter $\mathscr{F}$.
        \item[(e)] If $I$ is a (strong) $\H_Y$-ideal, then $I\cap R'$ is a (strong) $\H_{Y'}$-ideal.
        \item[(f)] $M'$ is a maximal (strong) $\H_{Y'}$-ideal \ff there is some maximal (strong) $\H_Y$-ideal such that $M' = M \cap R'$.
    \end{itemize}
\end{Pro}
\begin{proof}
    The proof is straightforward.
\end{proof}

\section{Some important classes of $\H_Y$-ideals, strong $\H_Y$-ideals and $Y$-Hilbert ideals}

In this section, we give propositions which generate a numerous class of $\H_Y$-ideals, strong $\H_Y$-ideals and $Y$-Hilbert ideals. Recall that, associated with each ideal $I$, there exists the ideal $m(I)=\{a\in R: a=ai \text{ for some } i\in I\} = \bigcup_{i \in I} \An(1-i) $ and associated with each prime ideal $ P $, there is the ideal $ O_P = \{‌ a \in R :‌ ab = 0 \text{ for some } b \notin P‌ \} = \bigcup_{a \notin P} \An(a) $. $ m(I) $ and $ O_P $ are called the quasi-regular part of $I$ and  the $P$ component of the zero, respectively. Also an ideal $I$ of $R$ is called pure if $I=m(I)$. It is easy to check that when a union of (strong) $\H_Y$-ideals is an ideal, then the union is also a (strong) $\H_Y$-ideal. We refer to \cite{mason1973z,aliabady2017regular,aliabad2017spectrum} for more detailed information about these classes of ideals. The following facts show that if the zero ideal is a (strong) $\H_Y$-ideal, then $Ann(I)$, $m(I)$ and $O_P$ are (strong) $\H_Y$-ideals, where $I$ and $P$ are an arbitrary ideal and a prime ideal of $R$, respectively. 

\begin{Pro}
    Suppose that $Y  \sub \Sp(R)$. If $J$ is a strong $\H_Y $-ideal, then $(J:I)$ is a strong $\H_Y $-ideal, for every ideal $I$ of $R$. The same assertions hold for $\H_Y $-ideals and $Y$-Hilbert ideals.
    \label{(J:I)}
\end{Pro}
\begin{proof}
	Suppose that $F$ is a finite subset of $(J:I)$ and $h_Y(F)\sub h_Y(a)$. For each $i\in I$
	\[ h_Y(Fi)=h_Y(F)\cup h_Y(i)=h_Y(a)\cup h_Y(i)=h_Y(ai)\]
	since $Fi$ is a finite subset of $J$ and $J$ is a strong $\H_Y$-ideal, it follows that $ai\in J$, thus $a\in(J:I)$. Using Proposition \ref{strong}, concludes that $(I:J)$ is a strong $\H_Y$-ideal.
\end{proof}

\begin{Def}
    Let $ Y  \sub \Sp(R)$ and $ k(Y)=\langle 0\rangle  $. By a minimal (strong) $\H_Y$-ideal we mean a non-zero (strong) $\H_Y$-ideal which contains no (strong) $\H_Y$-ideal except $ \langle0\rangle $.
\end{Def}

Recall that a ring $R$ is called Gelfand, if every prime ideal is contained in a unique maximal ideal. Also, a ring $R$ is called weakly regular, if every non-zero ideal contains a non-zero idempotent element.

\begin{Pro}
	Let $ Y  \sub \Sp(R)$ and $ k(Y)=\Ge{0}$. If $R$ is either a semiprimitive Gelfand ring or a weakly regular ring then the minimal $\H_Y$-ideals, the minimal strong $\H_Y$-ideals and the minimal $Y$-Hilbert ideals coincide.
\end{Pro}
\begin{proof}
     Let $I$ be a minimal $\H_Y$-ideal, minimal strong $\H_Y$-ideal or minimal $\H_Y$-ideal. If $R$ is a semiprimitive Gelfand ring $R$, then since $I$ is a non-zero ideal,by \cite[Thereom 3.2]{aliabad2017spectrum} $m(I)$ is non-zero ideal. By \cite[Propostion 2.1]{aliabad2017spectrum} $m(I)=\bigcup_{i\in I}\An(1-i)$. Therefore, $\Ge{0}\ne\An(1-i)\sub m(I)\sub I$, for some $i\in I$. Also if $R$ is a weakly regular ring, then the non-zero ideal $I$, contains an ideal of the form $\Ge{e}=\An(1-e)$, where $e$ is the non-zero idempotent element of $I$. So in the both of these rings we have $\An(x)\sub I$, for some $x\in R$. By Lemma \ref{introduction}, $\An(x)=kh_Y^c(x)$, hence $\An(I)$ is a $Y$-Hilbert ideal and so is (strong) $\H_Y$-ideal. Consequently, by the minimality of $I$, $I=\An(x)$ and we are done.
\end{proof}

The following proposition shows that a considerable class of ideals are strong $\H_Y$-ideal. Recall that for every multiplicatively closed subset $A$ and any ideal $I$ of a ring $R$ with $A\cap I=\tohi$, we can define  the  ideal  $I_A=\{r\in R: ra\in I$ for some $a\in A\}=\bigcup_{a\in A}(I:a)=\sum_{a\in A}(I:a).$

\begin{Pro}
    Suppose that $Y\sub\Sp(R)$. If $A$ is multiplicatively closed set and $I$ is a (strong) $\H_Y$-ideal of $R$ with $A\cap I=\tohi$, then $I_A$ is a (strong) $\H_Y $-ideal.
\label{opmi}
\end{Pro}
\begin{proof}
    Since $A\cap I=\emptyset$, $1\notin\bigcup_{a\in A}(I:a)=I_A$ is a proper ideal. By Proposition \ref{(J:I)}, $(I:a)$ is a (strong) $\H_Y$-ideal, for every $a\in A$. Clearly, $\{(I:a)\}_{a\in A}$ is a directed family of (strong) $\H_Y$-ideals and since the union of a directed family of (strong) $\H_Y$-ideals is also a (strong) $\H_Y$-ideal, it follows that $I_A=\bigcup_{a\in A}(I:a)$ is a (strong) $\H_Y$-ideal.
\end{proof}

\begin{Rem}
	\label{OPMI}
	Suppose that  $Y\sub \Sp(R)$, $k(Y)=\langle 0\rangle $ and $A$ is a multiplicatively closed subset  of  $R$. Then we can define the ideal ${\langle 0\rangle}_A=0_A=\{r\in R: ra=0$ for some $a\in A\}=\bigcup_{a\in A}\An(a)=\sum_{a\in A}\An(a)$. Since in this case $\langle 0\rangle$ is a (strong) $\H_Y$-ideal, $ 0_A $ is always a (strong) $\H_Y$-ideal, by Proposition \ref{opmi}. Some of the most important cases are the ideals $O_P=0_{R \setminus P}$ and $m(I)=0_{1+I}$ where $1+I=\{1+i: i\in I \}$. On the other words, if $Y\sub \Sp(R)$ and $k(Y)=\langle 0\rangle $, then the quasi-pure part (the zero-component) of every ideal (prime ideal) of $R$ is a strong $\H_Y$-ideal. Consequently  every pure ideal  is a strong $\H_Y$-ideal. Recall that an element $a$  of $R$ is called (Von Neumann) regular if $a=a^2b$, for some $b\in R$; an ideal $I$ is said to be regular, if every element of $I$ is regular and $R$ is called regular if each elements of $R$ is regular.  It is easy to see that every regular ideal is a pure ideal. Thus every minimal ideal, every summand of any ring and the socle of a reduced ring are pure (for example see \cite{aliabady2017regular}), hence they all are strong $\H_Y$-ideal.
\end{Rem}

\begin{Rem}
    In $C(X)$ if either $\Ma(C(X))\sub Y$ or $\Mi(C(X))\sub Y$, then $k(Y)=\langle 0\rangle$ is a strong $\H_Y$-ideal. Hence every minimal prime ideal is a strong $\H_Y$-ideal. Thus for every $A\sub \beta X$, $O^A$ which is an intersection of minimal prime ideals is a strong $\H_Y$-ideal. Thus, if $A$ is a round subset of $\beta X$ (i.e., from $A\sub cl_{\beta X} Z(f)$, it follows that $A\sub int_{\beta X} cl_{\beta X} Z(f)$), then $M^A$ is a strong $\H_Y$-ideal, too. By \cite[7E]{gillman1960rings}, $C_K(X)=O^{\beta X\setminus X}$ and by \cite[Theorem 3.1]{johnson1973functions}, $C_{\psi}(X)=O^{\beta X \setminus \nu X}$, so $C_K(X)$ and $C_\psi(X)$ are strong $\H_Y$-ideals.
\end{Rem}

In the sequel we focus on answering this question that ``What happens when all the ideals of a ring are either strong $\H_Y$-ideals or $\H_Y$-ideals?'', which gives another characterizations of regular rings. First we give the following lemma which is easy to prove.

\begin{Lem}
    Suppose that $ Y \subseteq \Sp(R) $. Then every finitely generated strong $\H_Y$-ideal of $R$ is a $ Y $-Hilbert ideal. Also, if $a \in R$, then the following  are equivalent.
    \begin{itemize}
    	\item[(a)] $ \langle a \rangle$ is an $ \H_Y $-ideal.
    	\item[(b)] $ \langle a\rangle $ is a strong $ \H_Y $-ideal.
    	\item[(c)] $ \langle a \rangle $ is $Y$-Hilbert ideal.
    \end{itemize}
\label{strong H ideal=Y-Hilbert ideal}
\end{Lem}

\begin{Pro}
    Let $Y\sub \Sp(R)$, then the following statements are equivalent:
    \begin{itemize}
        \item[(a)] Every ideal of $R$ is a strong $\H_Y $-ideal.
        \item[(b)] Every finitely generated ideal of $R$ is a strong $\H_Y $-ideal.
        \item[(c)] Every finitely generated ideal of $R$ is a $Y $-Hilbert ideal.
        \item[(d)] Every ideal of $R$ is an $\H_Y $-ideal.
        \item[(e)] Every principal ideal of $R$ is an $\H_Y $-ideal.
        \item[(f)] Every principal ideal of $R$ is a strong $\H_Y$-ideal.
        \item[(g)] Every principal ideal of $R$ is a $Y $-Hilbert ideal.
        \item[(h)] $k(Y)=\langle 0\rangle $ and $R$ is a regular ring.
        \item[(k)] $k(Y)=\langle 0\rangle $ and every essential ideal of $R$ is a strong $\H_Y $-ideal.
        \item[(l)] $k(Y)=\langle 0\rangle $ and every  essential ideal of $R$ is an $\H_Y $-ideal.
    \end{itemize}
    \label{H-ideal-regular}
\end{Pro}
\begin{proof} (a) $\RTA$ (b). It is clear.

    The implications (b) $\RTA$ (c) $\RTA$ (d) are straightforward.
    
    (d) $\RTA$ (e). It is evident.
    
    (e) $\RTA$ (f) $\RTA$ (g). They follow from Lemma \ref{strong H ideal=Y-Hilbert ideal}.

    (g) $\RTA$ (h). By the hypothesis the zero ideal is a $Y$-Hilbert and this implies that $k(Y)=\langle 0\rangle $. Also by the assumption, every ideal of $R$ is semiprime and consequently $R$ is a regular ring.

    (h) $\RTA$ (a). Sine $R$ is regular, for every ideal $I$ of $R$ we have $I=m(I)$ and so by Remark \ref{OPMI}, the result follows.

    (a) $\RTA$ (k)  $\RTA$ (l). They are trivial.

    (l) $\RTA$ (h). It is well-known and easy to be verified that if in  a reduced ring every essential ideal is a semiprime ideal then it is a regular ring. So by Lemma \ref{semiprime}, we are done.
\end{proof}

In the above proposition if we take either $\Ma(R)\sub Y$ or $\Mi(R)\sub Y$, then we can add the assertions: ``every ideal of $R$ is a $Y$-Hilbert ideal'' and `` Every essential ideal of $R$ is a  $Y$-Hilbert ideal''. In the following example we show that this is not true in general.

\begin{Exa}
    Suppose that $ R = C(\N)$, $Y  = {\cal{B}}(R)$ and $M$ is a maximal ideal of $R$. Since the zero ideal of $R$ is a fixed-place ideal, by \cite[Theorem 4.7]{aliabad2017Bourbarki}, it follows that there is an ultrafilter $\mathscr{U}$ on $Y$ such that $M = \J({\mathscr{U}}) = \{ a \in R : h_Y (a) \in \mathscr{U} \}$. Set $ \mathscr{F} = \mathscr{U} \cap \H_Y  $, then $ \mathscr{F} $ is a $\H_Y $-filter on $Y $ and
    \begin{align*}
    \H_Y^{-1}(\mathscr{F}) & = \{ a \in R : h_Y (a) \in \mathscr{F} \} \\
    					   & = \{ a \in R : h_Y (a) \in \mathscr{U} \cap \H_Y  \} \\
    					   & = \{ a \in R : h_Y (a) \in \mathscr{U} \} \\
    					   & = \J ( \mathscr{U} ) = M.
    \end{align*}
    Thus $M$ is a strong $\H_Y $-ideal. Since $R$ is regular ring, every ideal of $R$ is an intersection of maximal ideals and therefore every ideal is a strong $\H_Y $-ideal. But, if $M$ is a free maximal ideal, then $M$ is not a $Y $-Hilbert ideal.
    \label{ex 5.11}
\end{Exa}

\begin{Cor}
    Let $Y$ be a finite subset of $\Sp(R)$. If $I$ is an ideal of $R$, then the following statements are equivalent.
    \begin{itemize}
        \item[(a)] $I$ is an $\H_Y$-ideal.
        \item[(b)] $I$ is a strong $\H_Y$-ideal.
        \item[(c)] $I$ is a $Y$-Hilbert ideal.
    \end{itemize}
\end{Cor}
\begin{proof}
    It suffices to show (a) $\RTA$ (c). To see this, suppose that $I$ is an $\H_Y$-ideal. Since $Y$ is finite, by prime avoidance theorem, there exists some $x\in I\set\cup_{Q\in h_Y^c(I)}Q$. Thus, clearly, $h_Y(x)\sub h_Y(I)$ and so we have $kh_Y(I)\sub kh_Y(x)\sub I\sub kh_Y(I)$. Therefore, $I=kh_Y(I)$ and so $I$ is a $Y$-Hilbert ideal.
\end{proof}

\section {Operations on $\H_Y$-ideals, strong $\H_Y$-ideals and $Y$-Hilbert ideals}

As the title of this section shows, it is devoted to considering quotients, products, homomorphic images of $\H_Y$-ideals, strong $\H_Y$-ideals and $Y$-Hilbert ideals.

We shall note that, a product of $\H_Y$-ideals (resp., strong  $\H_Y$-ideals  and $Y$-Hilbert ideals) is not necessarily an $\H_Y$-ideal (resp., a strong  $\H_Y$-ideal  and a $Y$-Hilbert ideal). For instance, if we set $R=\Z$ and $\Ma(R)\sub Y\sub \Sp(R)$ then for every prime number $p$, the ideal $J=p\Z$ is a strong  $\H_Y$-ideal while $J^2=p^2\Z$ is not even a semiprime ideal. In general we have the following proposition.

\begin{Pro}
    Let $R$ be a ring and $\{J_i\}_{i=1}^{n}$ be a finite family of strong $\H_Y$-ideals  of $R$, then $\prod_{i=1}^{n}J_i$ is a strong $\H_Y$-ideal \ff $\bigcap_{i=1}^{n}J_i=\prod_{i=1}^{n}J_i$. The same statements hold for $\H_Y $-ideals and $Y$-Hilbert ideals.
\end{Pro}
\begin{proof}
    By Lemma \ref{semiprime}, it is clear.
\end{proof}

Let $f:R\to R'$ be a ring homomorphism and $I$ and $J$ be ideals of $R$ and $R'$, respectively. Then $I^e$ and $J^c$  denote the extension and  the contraction of the ideals $I$ and $J$, (i.e., $\langle f(I)\rangle$ and $f^{-1}(J)$), respectively.  In the following proposition we study the contraction of (strong) $\H_Y$-ideals and $Y$-Hilbert ideals under a ring homomorphism.

\begin{Pro}
	Let  $f:R\to R'$ be a ring homomorphism, $X\sub \Sp(R)$ and $Y\sub \Sp(R')$.  Every strong $\H_Y $-ideal  of $R'$ contracts to a strong $\H_X $-ideal  of $R$ \ff every $P\in Y$ contracts to a strong $\H_X $-ideal. The same statements hold for $\H_Y $-ideals and $Y$-Hilbert ideals.
	\label{h-ideal homo}
\end{Pro}
\begin{proof}
	($\RTA$). Suppose  that $ J $ is a strong $ \H_Y $-ideal of $ R' $. If $ F_1 $ and $ F_2 $ are two arbitrary subsets of $ R $ which $ h_X(F_1) = h_X(F_2) $ and $ F_1 \subseteq J^c $, then
	\begin{align*}
	P‌ \in h_Y(f(F_1)) & \LRTA \quad f(F_1)‌ \subseteq P  \quad \LRTA \quad F_1 \subseteq P^c \\
					   & \LRTA \quad P^c \in h_X(F_1) \quad \LRTA \quad P^c \in h_X(F_2) \\
					   & \LRTA \quad F_2 \subseteq P^c \quad \LRTA \quad f(F_2) \subseteq P \\
					   & \LRTA \quad P‌\in h_Y(f(F_2)). 
	\end{align*}
	So $ h_Y(f(F_1)) = h_Y(f(F_2)) $ and $ f(F_1) \subseteq J $, hence $ f(F_2) \subseteq J $, by Proposition \ref{strong}. Thus $ F_2 \subseteq J^c $ and this implies that $‌J^c$ is a strong $ \H_Y $-ideal, by Proposition \ref{strong}.
	
	($\LTA$). It is clear.
\end{proof}

\begin{Cor}
    Let $I\sub J$ be a pair of ideals of $R$, $Y\sub\Sp(R)$ and $Y/I=\{P/I: P\in h_Y(I)\}$. Then $J/I$ is a strong $\H_{\frac{Y}{I}} $-ideal \ff $J$ is a strong $\H_Y $-ideal. Also, supposing that $I_{\lambda}$ is an ideal of $R$, for every $\lambda \in \Lambda $, if $\sum_{\lambda \in \Lambda}I_{\lambda}$ is a direct sum and a strong $\H_Y$-ideal, then $I_{\lambda}$ is a strong $\H_Y$-ideal, for every $ \lambda \in \Lambda $. The same statements hold for $\H_Y $-ideals  and $Y$-Hilbert ideals.
\end{Cor}

The following corollaries show  the relation between the strong $\H_Y $-ideals of two different subspaces of $\Sp(R)$. Note that the same statements hold for the $\H_Y $-ideals and the $Y$-Hilbert ideals.

\begin{Cor}
	Let $X,Y \sub \Sp(R)$. Then we have the following facts:
	\begin{itemize}
		\item[(a)] Every element of $X$ is a strong $\H_Y$-ideal \ff every strong $\H_X$-ideal  is a strong $\H_Y$-ideal.
		\item[(b)] If $X\sub Y$, then every strong $\H_X$-ideal is a strong $\H_Y$-ideal.
		\item[(c)] If $X\sub Y$ and every element of $Y$ is a strong $\H_X$-ideal, then the strong $\H_X$-ideal and the strong $\H_Y$-ideals coincide.
	\end{itemize}
	\label{Y and X -ideal}
\end{Cor}
\begin{proof}
	If we take the identity mapping from $(R,Y)$ to $(R,X)$ and apply Proposition \ref{h-ideal homo}, then they conclude.
\end{proof}

\begin{Cor}
    Let $X,Y \sub \Sp(R)$, $I_\circ = k(X) \sub k(Y)$ and $X \sub \Mi(I_\circ)$. Every strong $\H_X$-ideal ($\H_X $-ideal) is a strong $\H_Y$-ideal ($\H_Y $-ideal) \ff $k(X) = k(Y)$.
\end{Cor}
\begin{proof}
    We just prove the  part concerned with the strong $ \H_Y $-ideal. The part concerned  with the $\H_X$-ideal has a same proof.

    $\RTA$) Clearly, $I_\circ$ is a strong $\H_X$-ideal and therefore $I_\circ$ is a strong $\H_Y$-ideal. Again, $k(Y)$ is the smallest strong $\H_Y$-ideal, since $I_\circ \sub k(Y)$, we conclude that  $k(Y) = I_\circ =k(X)$.

    $\LTA$) Since $I_\circ = k(Y)$, $ I_\circ $ is a strong $\H_Y$-ideal and therefore every element of $\Mi(I_\circ)$ is a strong $H_Y$-ideal, hence every element of $X$ is a strong $\H_Y$-ideal and therefore each strong $\H_X$-ideal is a strong $\H_Y$-ideal, by Corollary \ref{Y and X -ideal}.
\end{proof}

\begin{Pro}
	Let $A$ be a multiplicatively closed subset of $R$ and $f:R\to A^{-1}R$ be the natural ring homomorphism. If $I$ is a (strong) $\H_Y$-ideal, then $I^{ec}$ is a (strong) $\H_Y$-ideal, too.
\end{Pro}
\begin{proof}
	It is easy to see that $I^{ec}=I_A$ and so by Proposition \ref{opmi} we are done.	
\end{proof}

\section{Certain (strong) $\H_Y$-ideals over or contained in an ideal  }

This section is about the particular  (strong) $\H_Y$-ideals related to an ideal. First we study the maximal (strong) $\H_Y$-ideals, then the smallest  (strong) $\H_Y$-ideal containing an ideal are characterized. As we will see that some of the results hold for $Y$-Hilbert ideals too. For convenience we use some notations. Let $E$ be a partial ordered set. By $maxl(E)$, we mean the set of all maximal elements of $E$. Also if $R$ is a ring, $Y\sub\Sp(R)$ and ${\cal{A}}\sub{\cal{I}}(R)$, we denote by $S\H_Y({\cal{A}})$ ($PS\H_Y({\cal{A}})$) the set of all strong $H_Y$-ideals (proper strong $H_Y$-ideals) of $A$. For $\H_Y$-ideals we use the notations $\H_Y({\cal{A}})$ and $P\H_Y({\cal{A}})$, respectively. By $[I,J]$ we mean $\{K\in{\cal{I}}(R):~I\sub K\sub J\}$; and by $\downarrow I$ and $\uparrow I$ we mean $\{K\in{\cal{I}}(R):~K\sub I\}$ and $\{K\in{\cal{I}}(R):~I\sub K\}$, respectively.
It is straightforward to observe that the union of a chain of proper (strong) $\H_Y$-ideals is a proper (strong) $\H_Y$-ideal.

\begin{Pro}
Let $R$ be a ring and $Y\sub\Sp(R)$ and $I$ is a proper (strong) $\H_Y$-ideal of $R$. Then the following statements hold.
    \begin{itemize}
        \item[(a)]  For every ideal $J\supseteq I$,  $maxl(P\H_Y[I,J])\neq\emptyset$  ($maxl(PS\H_Y[I,J])\neq\emptyset$). In the particular, $maxl(P\H_Y(\uparrow I))\neq\emptyset$ ($maxl(PS\H_Y(\uparrow I))\neq\emptyset$) and for every ideal $J\supseteq k(Y)$, $maxl(P\H_Y(\downarrow J))\neq\emptyset$ ($maxl(PS\H_Y(\downarrow J))\neq\emptyset$).
        \item[(b)] Let $Y\sub \Sp(R)$ and $P$ be a prime ideal containing $k(Y)$. Then $ maxl(\H_Y(\downarrow P)) $ and $ maxl(S\H_Y(\downarrow P)) $ are contained in $ \Sp(R)$.
        \item[(c)] If $k(Y)=\langle 0 \rangle$, then every prime ideal of $R$ is either a (strong) $\H_Y$-ideal or contains a maximal (strong) $\H_Y$-ideal which is a  prime (strong) $\H_Y$-ideal.
    \end{itemize}
    \label{maxiaml storng ideal}
\end{Pro}
\begin{proof}
	We just prove the  part concerned with the strong $ \H_Y $-ideal. The part concerned  with the $\H_X$-ideal has a same proof.
	
	(a). By using Zorn's lemma, it implies immediately.
	
	(b). If $P$ is an $\H_Y$-ideal, then it is clear. Now suppose that $P$ is not an $\H_Y$-ideal, by part (a), $Q \in maxl(\H_Y(\downarrow P))$. Since $P$ is not an $\H_Y$-ideal, by Corollary \ref{corollary}, $P \notin \Mi(Q)$, so $Q' \in \Mi(Q)$ exists such that $Q' \sub P$. Now Corollary \ref{corollary}, deduces that $Q'$ is an $\H_Y$-ideal, so $Q=Q'$ is prime, by maximality of $Q$.
	
	(c). It is clear by part (b). 
\end{proof}

In the following example we show that $ maxl(PS\H_Y[I,J]) $ need not be a proper maximal strong $ H_Y $-ideal, even if $J$ is a maximal ideal.

\begin{Exa}
    Suppose that $ R = \R[x,y]$, $I=\langle x-1\rangle$, $J=\langle y\rangle$, $K=\langle x-1,y\rangle$, $M=\langle x,y\rangle$ and $Y=\{J, K\}$. It is clear that $Y \subseteq \Sp(R)$ and it is easy to show that $maxl (PS\H_Y[I \cap J , M ])‌= \{J\}$, whereas $J$ is not a proper maximal strong $ \H_Y $-ideal, because $K$ is a strong $ \H_Y $-ideal that properly contains $J$.
\end{Exa}

Using Theorem \ref{minimal of I},  Proposition \ref{Primary property} and Proposition \ref{Minimal filter and ideal}, one can obtain the following corollary straightforward.

\begin{Cor}
    Let $Y\sub \Sp(R)$ and $\Rad(R)=k(Y)$. Every $P \in \Mi(R)$ is a strong $\H_Y $-ideal and therefore there is some minimal prime $\H_Y $-filter $\mathscr{P}$ such that $ \H_Y^{-1}(\mathscr{P}) = P$.
\end{Cor}

\begin{Def}
    Let $Y  \sub \Sp(R)$. It is obvious that  the intersection of any family of $\H_Y $-ideals (resp., strong $\H_Y $-ideals and $Y$-Hilbert ideals) is an $\H_Y $-ideal (resp., a strong $\H_Y $-ideal  and $Y$-Hilbert ideals). According to this fact, the smallest $\H_Y $-ideal (resp., strong $\H_Y $-ideal  and $Y$-Hilbert ideal) containing an arbitrary ideal $I$ exists. We denote it by $I_{\H_Y}$ (resp., $I_{S\H_Y}$ and $kh_Y(I)$) which is the intersection of  $\H_Y $-ideals (resp., strong $\H_Y $-ideals  and $Y$-Hilbert ideals) containing $I$. If there is not any ambiguity we use $I_{\H}$ (resp., $I_{S\H}$) instead of $I_{\H_Y}$ (resp., $I_{S\H_Y}$).
\end{Def}

Clearly, if $Y=\Ma(R)$, then the concepts of $I_\H$ and $I_{S\H}$ coincide with the concepts of $I_z$ and $I_{sz}$, respectively. See \cite{aliabad2011sz} and \cite{mason1989prime} for more detailed information about these concepts. Also, if $Y=\Mi(R)$, then the concepts of $I_\H$ and $I_{S\H}$ coincide with the concepts of $I_{z^\circ}$ (also known as $I_\circ$ and $I^\circ$) and  $I_{sz^\circ}$ (also known as $\zeta(I)$-ideal), respectively. We refer to \cite{aliabad2011sz}, \cite{azarpanah1999z}, \cite{azarpanah2000ideals}, and \cite{mason1973z}, for more information about these concepts. Finally if $Y=\Sp(R)$, then the concepts of $I_\H$ and $I_{S\H}$ and $\sqrt{I}$ coincide.  It is clear that $I_\H \subseteq I_{S\H}$.

\begin{Pro}
    Let $ Y\sub \Sp(R)$ and $I$ and $J$ be two ideals of $R$. Then the following statements hold:
    \begin{itemize}
        \item[(a)] $I_{S\H}=\H_Y^{-1}\H_Y(I)=\{a\in R:~\ex F\in{\bf F}(I),~ h_Y(F) \sub h_Y(a)\}=\{a\in R:~\ex F\in{\bf F}(I),~ kh_Y(a)\sub kh_Y(F)\}$ and if $R$ satisfies $h_Y$-property, then we have  $I_{S\H}=I_{\H} = \{a \in R : \ex~ b \in I \text{ such that } h_Y(b) \sub h_Y(a) \}$. 
        \item[(b)] $I_{S\H}=\sum_{F\in{\bf F}(I)}kh_Y(F)=\bigcup_{F\in{\bf F}(I)}kh_Y(F)$.
        \item[(c)] $ (IJ)_{\H} = I_{\H} \cap J_{\H} = (I\cap J)_{\H} $\big(resp., $(IJ)_{S\H} = I_{S\H} \cap J_{S\H} = (I\cap J)_{S\H}$\big).
        \item[(d)] $kh_Y(I)=\{a\in R:~\ex S\sub I,~ h_Y(S) \sub h_Y(a)\}$. Also $kh_Y(IJ) = kh_Y(I)\cap kh_Y(J)=kh_Y(I\cap J)$.
    \end{itemize}
\label{IH}
\end{Pro}
\begin{proof}
	(a). By Proposition \ref{Primary property}, $\H_Y^{-1}\H_Y(I)$ is a strong $\H_Y$-ideal containing $I$. Now, assume that $J$ is a strong $\H_Y$-ideal containing $I$, then $\H_Y^{-1}\H_Y(I)\sub \H_Y^{-1}\H_Y(J)=J$,  so the first equality holds. On the other hand, since $I_{S\H}$ is a strong $\H_Y$-ideal, the set $ H=\{a\in R:~\ex F\in{\bf F}(I),~ h_Y(F) \sub h_Y(a)\}$ is  a subset of $I_{S\H}$. Since $H$ contains $I$, to show the second equality it is enough to prove that $H$ is a strong $\H_Y$-ideal. Let $a,b\in H$, so there exist finite subsets $F_1$ and $F_2$ of $I$ such that $ h_Y(F_1) \subseteq h_Y(a) $ and $ h_Y(F_2) \subseteq h_Y(b) $, so
	\[ h_Y\big(F  _1\cup F_2\big)=h_Y(F_1)\cap h_Y(F_2)\sub h_Y(a)\cap h_Y(b)\sub h_Y(a+b). \]
	Since $F_1\cup F_2$ is finite, $a+b\in H$. Also, since $ h_Y(a)\sub h_Y(ra) $, for each $r\in R$, $H$ is an ideal. Now it is enough to show that $H$ is a strong $\H_Y$-ideal. Let $ h_Y(F)\sub h_Y(a) $, where $F=\{x_1,x_2,\ldots,x_n\}$ is a finite subset of $H$, then for each $1 \leqslant i \leqslant n$,  there exists a finite set $F_i\sub I$ such that $h_Y(F_i)\sub h_Y(x_i)$. Now we have that
	\[ h_Y\Big(\bigcup_{i=1}^{n}F_i\Big)=\bigcap_{i=1}^{n}h_Y(F_i)\sub \bigcap_{i=1}^{n}h_Y(x_i)=h_Y(F)\sub h_Y(a). \]
	Now since $ \bigcup_{i=1}^{n}F_i$ is a finite subset of $I$, we are done. It is clear that if $R$ satisfies in $h_Y$-property, then $I_{S\H}=I_{\H} = \{a \in R : \ex~ b \in I \text{ such that } h_Y(b) \sub h_Y(a) \}.$
	
	(b). Since $\{kh_Y(F)\}_{F\in{\bf F}(I)}$ is a directed set, it follows that $\bigcup_{F\in{\bf F}(I)}kh_Y(F)$ is an ideal and so $\sum_{F\in{\bf F}(I)}kh_Y(F)=\bigcup_{F\in{\bf F}(I)}kh_Y(F)$. Also, it is clear that $\bigcup_{F\in{\bf F}(I)}kh_Y(F)$ is a strong $\H_Y$-ideal containing $I$, so $\bigcup_{F\in {\bf{F}}(I)} kh_Y(F) = I_{S\H}$.
	
	(c). Obviously, since $IJ\sub I\cap J$, we have $(IJ)_\H \sub (I\cap J)_\H\sub I_\H\cap J_\H$ (resp., $(IJ)_{S\H}\sub (I\cap J)_{S\H}\sub I_{S\H}\cap J_{S\H}$). Now, suppose that $P$ is a prime $\H_Y$-ideal containing $(IJ)_\H$ (resp., a prime strong $\H_Y$-ideal containing $(IJ)_{S\H}$). Then clearly $I\sub P$ or $J\sub P$ and so $I_\H\sub P$ or $J_\H\sub P$ (resp., $I_{S\H}\sub P$ or $J_{S\H}\sub P$) and consequently, $I_\H\cap J_\H\sub P$ (resp., $I_{S\H}\cap J_{S\H}\sub P$). 

	(d). It is clear.
\end{proof}

Recall that a ring $R$ satisfies property $A$, if each finitely generated ideal of $R$ consisting of zero-divisors has a non-zero annihilator (equivalently every finitely generated ideal with a zero annihilator contains a non zero-divisor, known as Condition C in \cite{quentel1971compacite}). As it is stated in \cite{azarpanah2000ideals}, Noetherian rings, $C(X)$, regular rings satisfy property $A$. Clearly a proper ideal $I$ is contained in a proper (strong) $\H_Y$-ideal \ff $I_{\H}$ ($I_{S\H}$) is a proper ideal. Also according to Proposition \ref{Primary property}, $\H_Y^{-1}\H_Y (I)$ is a proper ideal of $R$ \ff  $\emptyset\notin\H_Y (I)$ (equivalently, $\H_Y (I)$ is a proper $\H_Y$-filter). It is also clear that every maximal ideal is a (strong) $\H_Y$-ideal \ff every proper ideal is contained in a proper (strong) $\H_Y$-ideal. For any $Y\sub \Sp(R)$ we have the following corollary of the above proposition which is an improvement of \cite[Theorem 1.21]{azarpanah2000ideals} with a totally different proof.

\begin{Cor}
	Let $R$ be a reduced ring, $Y\sub \Sp(R)$, $k(Y) = \Ge{0} $ and $R$ satisfies property $A$. Then any singular ideal $I$ (i.e., every element of $I$ is a zero-divisor) is contained in a proper strong $\H_Y$-ideal and therefore is contained in a proper $\H_Y$-ideal.
\end{Cor}
\begin{proof}
	It is sufficient to show that every element of $ I_{S\H} $ is a zero-divisor. Let $a\in  I_{S\H}$, thus $h_Y(F)\sub h_Y(a)$, for some finite set $F\sub I$, thus $kh_Y^c(F)\sub kh_Y^c(a)$, therefore by Lemma \ref{introduction}, $\An(F)\sub \An(a)$. Since $R$ satisfies property $A$ and $F$ consists of zero-divisors, it follows that $\An(a)\ne\Ge{0}$, that is, $a$ is a zero-divisor.
\end{proof}

If $ k(Y) = \Ge{0} $, then according to Lemma \ref{kykx=0}, Lemma \ref{hminterior} and Proposition \ref{IH}, we have the following characterization of $I_{sz^{\circ}}$.

\begin{Cor}
	Let $R$ be a ring, $Y\sub \Sp(R)$, $k(Y)=\Ge{0}$. Then $I_{sz^{\circ}} = \{ a \in R : \left( h_Y(F) \right) ^\circ\sub h_Y(a) \text{  for some finite } F\sub I\}$.
\end{Cor}

\begin{Cor}
	Let $R$ be a ring, $Y\sub \Sp(R)$, $k(Y)=\Ge{0}$ and $R$ satisfies the property $A$. Then the following facts hold:
	\begin{itemize}
		\item[(a)] Every maximal ideal consisting of zero-divisors is a (strong) $\H_Y$-ideal.
		\item[(b)] Every  ideal consisting of zero-divisors is contained in a maximal (strong) $\H_Y$-ideal which is a prime ideal.
	\end{itemize}
\end{Cor}

\begin{Rem}
	With a method similar to \cite[Thoerem 1.21]{azarpanah2000ideals}, we can observe that if $I$ is an ideal of $R$ and we set $I_0=I$, $I_1=\sum_{a\in I_0}kh_Y(a)$, $I_\alpha=\sum_{a\in I_\beta}kh_Y(a)$ for a nonlimit ordinal $\alpha=\beta +1$ and $I_\alpha=\bigcup_{\beta \le \alpha} I_{\beta}$, for a limit ordinal $\alpha$, then the smallest ordinal $\alpha$ that $I_\alpha=I_\gamma$, for every $\gamma \geq \alpha$, is exactly $I_\H$.
\end{Rem}

\begin{Pro}
	Let $R$ be  a ring, $I$  is an arbitrary ideal of $R$ and $ Y\sub \Sp(R)$. The  following statements hold.
	\item[(a)] If $k(Y)=\Ge{0} $, then  $(m(I))_{\H} = (m(I))_{S\H} = m(I) $.
	\item[(b)] If $\Ma(R)\sub Y$, then  $m(I)=m(I_{\H})=m(I_{S\H})=m(kh_Y(I)) $.
\end{Pro}
\begin{proof}
	(a). It is clear from Remark \ref{OPMI}.
	
	(b). It is shown in \cite[Remark 2.6]{aliabad2017spectrum} that when $\Ma(R)\sub Y$, then $m(kh_Y(I))=m(I)$, for every ideal $I$ and since clearly $I\sub I_{\H}\sub I_{S\H}\sub kh_Y(I)$, it follows that $m(I)=m(I_{\H})=m(I_{S\H})=m(kh_Y(I))$.
\end{proof}

The condition $\Ma(R)\sub Y$ is necessary for the equalities of part (b) of the above proposition. For instance, if $Y=\Mi(R)$ and $M$ is a maximal ideal containing a non zero-divisor, then $m(M)\sub M\ne R =m(M_{\H_Y})=m(M_{S\H_Y})$.  Also, as we see  in Example \ref{ex 5.11}, there exists a ring $R$, $Y\sub \Sp(R)$ and a maximal ideal  $M\notin Y$ such that $M$ is a strong $\H_Y$-ideal. Hence, in this case $m(M_{S\H_Y})=m(M)\neq R=m(R)=m(kh_Y(M))$. 

As a corollary of the above proposition we have the following proposition which gives more facts about $I_{\H}$, $I_{S\H}$ and $kh_Y(I)$.

\begin{Pro}
	Let $X,Y \sub \Sp(R)$. If $ I $ is an ideal of $R$ and $ n \in \N $, then
	\begin{itemize}
		\item[(a)] $I^n\sub I \sub \sqrt{I} \sub I_{\H_Y}  \sub I_{S\H_Y} \sub kh_Y(I)$.
		\item[(b)] $(I^n)_{\H_Y}=(\sqrt{I})_{\H_Y}=I_{\H_Y}$,  $(I^n)_{S\H_Y}=(\sqrt{I})_{S\H_Y}=I_{S\H_Y}$ and $kh_Y(I^n)=kh_Y(\sqrt{I})=kh_Y(I)$.
		\item[(c)] If every element of $Y$ is an  $\H_X$-ideal (resp., strong $\H_X$-ideal and $X$-Hilbert ideal), then $I_{\H_X}\sub I_{\H_Y}$ (resp., $I_{S\H_Y}\sub \H_Y^{-1}\H_Y(I)$ and $kh_X(I)\sub kh_Y(I)$).
	\end{itemize}
	\label{Opearating _H}
\end{Pro}
\begin{proof}
	(a). It follows form Lemma \ref{semiprime}. 
	
	(b). By Proposition \ref{IH} and part (a) it is straightforward.  
	
	(c). It follows from Corollary \ref{Y and X -ideal}.
\end{proof}

Supposing $ R = \R[x,y] $, $ Y = \{ \Ge{x} , \Ge{y} \} $, then $ \Ge{x,y} = \Ge{x} + \Ge{y}  $ is not an $ \H_Y $-ideal, so the sum of a two strong $ \H_Y $-ideals, need not be an $ \H_Y $-ideal. One can easily see that the sum of a family of strong  $\H_Y$-ideals $\{I_{\lambda} \}_{\lambda\in \Lambda}$ is a strong  $\H_Y$-ideal   \ff  $(\sum_{\lambda\in \Lambda}I_{\lambda})_{S\H}= \sum_{\lambda\in \Lambda}(I_{\lambda})_{S\H}$. In addition, we can see that    $(\sum_{\lambda\in \Lambda}I_{\lambda})_{S\H}= (\sum_{\lambda\in \Lambda}(I_{\lambda})_{S\H})_{S\H}$. Also, if $\Ma(R)\sub Y$, then $\sum_{\la\in\La}I_\la=R$ \ff $\sum_{\la\in\La}(I_\la)_{S\H}=R$. To see this, suppose that $\sum_{\la\in\La}I_\la\neq R$, then by the hypothesis there is a strong $\H_Y$-ideal containing $\sum_{\la\in\La}I_\la$ and so $(\sum_{\la\in\La}I_\la)_{S\H_Y}\neq R$. Therefore, $\sum_{\la\in\La}(I_\la)_{S\H_Y}\sub (\sum_{\la\in\La}(I_\la)_{S\H_Y})_{S\H_Y}=(\sum_{\la\in\La}(I_\la))_{S\H_Y}\neq R$. We shall note that the same statements hold for the case  $\H_Y$-ideals.



\end{document}